\documentclass{article}
\usepackage{amsthm}
\usepackage{amsmath}
\usepackage{amssymb}
\usepackage{amsfonts}
\usepackage[margin=1.5in]{geometry}

\newtheorem{theorem}{Theorem}[section]
\newtheorem{corollary}[theorem]{Corollary}
\newtheorem{proposition}[theorem]{Proposition}
\newtheorem{lemma}[theorem]{Lemma}
\newtheorem{observation}[theorem]{Observation}
\newtheorem{remark}[theorem]{Remark}

\newtheorem{sstheorem}{Theorem}[subsection]
\newtheorem{ssproposition}[sstheorem]{Proposition}
\newtheorem{sslemma}[sstheorem]{Lemma}
\newtheorem{ssremark}[sstheorem]{Remark}

\newcommand{\ssection}[1]{%
  \section[#1]{\centering\normalfont\scshape #1}}
\newcommand{\ssubsection}[1]{%
  \subsection[#1]{\raggedright\normalfont\itshape #1}}

\author{Holly Krieger}
\date{\today}
\title{Primitive prime divisors in the critical orbit of $z^d+c$}

\begin{document}

\maketitle

\begin{abstract} 
We prove the finiteness of the Zsigmondy set associated to the critical orbit of $f(z) = z^d+c$ for rational values of $c$ by uniformly bounding the size of the Zsigmondy set for all $c \in \mathbb{Q}$ and all $d \geq 2$.  We prove further that there exists an effectively computable bound $M(c)$ on the largest element of the Zsigmondy set, and that under mild additional hypotheses on $c$, we have $M(c) \leq 3$.
\end{abstract}

\ssection{Introduction} \label{Introduction}
The {\it Zsigmondy set} of a sequence $\{ a_ n \}$ of integers is the set of indices $n$ for which $\{ a_n \}$ fails to be divisible by a {\it primitive prime divisor}; i.e., $n$ is in the Zsigmondy set if for every prime dividing $a_n$, there exists $1 \leq k < n$ such that $p \mid a_k$.  The notion of a Zsigmondy set originated from a theorem of Bang \cite{Bang:TU} and Zsigmondy \cite{Zsigmondy:ZTDP} characterizing the Zsigmondy set of the sequence $\{a^n - b^n\}$ for coprime integers $a > b > 0$.  This type of result was then extended to the setting of Lucas sequences \cite{Carmichael:OTNF}, \cite{Schinzel:PDOT}, elliptic divisibility sequences \cite{Silverman:WCAT}, and sequences associated to the iteration of rational functions \cite{IngramSilverman:PDIA}.  In \cite{IngramSilverman:PDIA}, Ingram and Silverman proved that the Zsigmondy set of a sequence associated to iteration of certain rational functions is finite.  Rice \cite{Rice:PPDI} has proved the finiteness of the Zsigmondy set associated to the critical orbit of $f(z) = z^d+c$ for any $c \in \mathbb{Z}$, and Doerksen and Haensch \cite{DoerksenHaensch:PPDI} explicitly characterized the Zsigmondy set in this case. \\

In this article, we study the Zsigmondy set of the critical orbit of $z^d+c$ for $c \in \mathbb{Q}$.  Supposing that $c = \frac{a}{b}$ in lowest terms, the $n$th iterate of 0 can by induction be written 

$$f^n(0) = \frac{a_n}{b^{d^{n-1}}}$$
for some integer $a_n$ coprime to $b$.  Consequently, one sees that the critical orbit is infinite unless $c \in \{0, -1, -2\}$.  In this paper, we resolve the question of the finiteness of the Zsigmondy set, which we denote by $Z(f, 0)$, finding a bound on the size of the Zsigmondy set which is uniform in both $d$ and $c$: 

\begin{theorem} \label{DAEff} Let $f(z) = z^d + c$ and $c \in \mathbb{Q}$ such that the critical orbit is infinite.  Then $\# \mathcal{Z}(f, 0) \leq 23.$ \\
\end{theorem}

The heart of this result is Mahler's \cite{Mahler:OTCF} refinement of Thue's precursor \cite{Thue:UAAZ} to Roth's theorem on rational approximations of $d$th roots of integers.  Utilizing the rapid growth of the denominator of $f^n(0),$ the existence of multiple sufficiently large elements of the Zsigmondy set would give multiple extremely good rational approximates to a certain algebraic integer, contradicting Mahler's result.   We also use a result of Bennett and Bugeaud \cite{BennettBugeaud:ERFR} on approximation of quadratic irrationals to establish the existence of an effective (though non-uniform) bound on the largest element of $\mathcal{Z}(f, 0)$:\\

\begin{theorem} \label{EffExists} Suppose $c \in \mathbb{Q}$, and $f(z) = z^d+c$ such that the critical orbit is infinite.  Then there exists an effectively computable constant $M(c)$ such that for all $n > M(c)$, $n \notin \mathcal{Z}(f,0)$.  
\end{theorem}

Though the proof of Theorem \ref{EffExists} generally yields large values of $M(c)$, it is the case that with various additional assumptions, $M(c)$ is quite small.  Namely, we have the following bounds:

\begin{theorem} \label{NREff} Let $f(z) = z^d + c$ with $d \geq 2$ and $c = \frac{a}{b} \in \mathbb{Q}$ in lowest terms.  If $d$ is odd, or $d$ is even and $c \notin (-2^{\frac{1}{d-1}}, -1),$ then we can take $M(c) = 2$.  Further, $\mathcal{Z}(f, 0)$ is empty unless $d = 2$ and $a + b = \pm 1$, or $c = \pm 1$.
\end{theorem} 

Even in the case when $d$ is even and $c \in (-2^{\frac{1}{d-1}}, -1)$, one can provide a lower bound for the recurrence of 0 for parameters $c$ which are not too close to any parameter with finite critical orbit.  We do this when $d = 2$, considering $c$ as a complex parameter for $f_c(z) = z^2 + c$.  In particular, given $n \geq 1,$ fix $\rho_n > 0$, and define $D(n, \rho_n)$ to be the set of complex parameters $c$ such that 0 lies in an attracting basin of a complex number $a$ with exact period $n$ satisfying $|(f_c^n)'(a)| \leq \rho_n$.  

\begin{theorem} \label{Mandelbrot}Define $D(n, \rho_n)$ as above with $\rho_n = \min \{ \frac{1}{4},  \frac{1}{2^{2^{n-2}}} \}$.  Write

$$S := \mathbb{C} - \bigcup_{n \in \mathbb{N}} D(n, \rho_n).$$
Then for all $c = \frac{a}{b} \in S$, writing $f(z) = z^2+c$, we can take $M(c) = 3$.
\end{theorem} 

The layout of this paper is as follows: in Section \ref{Prelim}, we establish preliminary lemmas which allow an arithmetic characterization of $n \in \mathcal{Z}(f, 0)$; in particular, $n \in \mathcal{Z}(f, 0)$ provides an upper bound on size of the numerator of $f^n(0)$.  Section \ref{Eff1} contains the proof of Theorem \ref{DAEff}, utilizing a result of Mahler to make effective the general notion that an iterate $f^n(0)$ with small numerator will yield a rational approximate to the $d$th root of $c$ which is too good.  The existence of an effective bound $M(c)$ is established in Section \ref{Eff2} using the same idea, and Theorem \ref{NREff} is established in Section \ref{tight} via the theory of canonical heights. In Section \ref{Mand}, Theorem \ref{Mandelbrot} is proved, using de Branges' theorem to find a lower bound for the numerator of $|f^n(c)|$ for those values of $c$ which are not too close to centers of hyperbolic components of the Mandelbrot set, obstructing $n \in \mathcal{Z}(f, 0)$ for $n > 3$.\\

{\bf Related Questions.} 
Though in the interest of length we have restricted ourselves to the rational case, the majority of these results have immediate analogues if we allow $c$ to be an algebraic number, and ask about the prime ideal divisors of the numerators of the ideals generated by the iterates of $0$.  Again it is not hard to show that $M(c) = 2$ in the integral case, but the non-integral case requires more machinery.  An application of quantitative Roth's theorem results such as \cite{Silverman:WCAT} will yield a bound on the size of the Zsigmondy set in this case, generalizing Theorem \ref{DAEff}, though the bound will no longer be uniform in $d$.  The non-recurrence statements of Theorem \ref{NREff} and \ref{Mandelbrot} are independent of choice of archimedean norm, and so can also be applied to compute $M(c)$ in these cases. \\

Zsigmondy questions of this sort also connect to broader problems in number theory and arithmetic dynamics.  Recently, Gratton, Nguyen, and Tucker have shown (personal communication) that the $abc$ conjecture implies a finite Zsigmondy set for the numerator sequence of any infinite orbit under rational iteration.  Silverman and Voloch have shown \cite{SilvermanVoloch:ALGC} that the Zsigmondy result of \cite{IngramSilverman:PDIA} can be used to prove that there is no dynamical Brauer-Manin obstruction for dimension 0 subvarieties under morphisms $\phi: \mathbb{P}^1(K) \rightarrow \mathbb{P}^1(K)$ of degree at least 2, while Faber and Voloch have utilized the Zsigmondy result of \cite{IngramSilverman:PDIA} in studying non-archimedean convergence of Newton's method \cite{FaberVoloch:OTNO}. \\ 

Another related area of interest is the question of the density of prime divisors of the critical orbit.  Jones \cite{Jones:TDOP} has shown that for $d=2$, when $c \in \mathbb{Z}$ is critically infinite, the density of primes $p$ dividing some element of the critical orbit is 0.  This is in spite of the result of \cite{DoerksenHaensch:PPDI} that for each $n \geq 3$ we have a primitive prime divisor, so one could ask whether we have the same phenomenon for $c \in \mathbb{Q}$.  Similarly, one can ask Zsigmondy questions about other sequences related to dynamical systems, and Faber and Granville \cite{FaberGranville:PFOD} have proven (barring an obvious obstruction) that for any fixed $\Delta \in \mathbb{N}$ and any $\phi \in \mathbb{Q}(z)$, the sequences of numerators of differences $\phi^{n+\Delta}(x)-\phi^n(x)$ have finite Zsigmondy set for any point $x \in \mathbb{Q}$ with infinite forward orbit. \\

\ssection{Preliminary Results.} \label{Prelim}

Throughout we write $f(z) = z^d + c$, where $d \geq 2$ and $c = \frac{a}{b} \in \mathbb{Q}$ in lowest terms, choosing $b$ positive.  

\begin{observation} With the above notation, the $n$th iterate $f^n(0)$ is written in lowest terms as

$$f^n(0) = \frac{a_n}{b^{d^{n-1}}},$$
for some $a_n \in \mathbb{Z} \setminus \{ 0 \}$ coprime to $b$.  Consequently, the critical orbit is infinite for all $b \geq 2$.
\end{observation}
We define the {\it Zsigmondy set} associated to $f$, $\mathcal{Z}(f,0)$, to be the set of indices $n \geq 2$ such that $a_n$ has no primitive prime divisor; i.e., for all primes $p$ dividing $a_n$, there exists $1 \leq k < n$ with $p \mid a_k$.  \\

The case $c \in \mathbb{Z}$ has been treated in \cite{DoerksenHaensch:PPDI}, rephrased here in our notation:

\begin{proposition} \label{DH} [Doerksen-Haensch]  Suppose $f(z) = z^d+c$ with $d \geq 2$ and $c \in \mathbb{Z}$ such that the critical orbit is infinite. Then $n \in \mathcal{Z}(f, 0) \Rightarrow n \leq 2$, and $\mathcal{Z}(f, 0)$ is empty unless $c = \pm 1$.
\end{proposition}

Our methods are different than those used in \cite{DoerksenHaensch:PPDI} and in fact will utilize the rapid growth of the denominators of the forward orbit.  Since Proposition \ref{DH} implies Theorems \ref{DAEff}, \ref{EffExists}, \ref{NREff}, and \ref{Mandelbrot} for integral values of $c$, we assume throughout that $b \geq 2$. \\

In this section, we quantify the statement $n \in \mathcal{Z}(f,0)$.  To begin, we note that the sequence $\{ a_n \}$ forms a rigid divisibility sequence: \\

\begin{lemma}  \label{RDS} Let $f(z)$ be as above.  Suppose $p$ is a prime which divides some element of the sequence $\{ a_n \},$ and let $k(p) \geq 1$ be the minimal natural number such that {\rm ord}$_p(a_{k(p)}) > 0$.  Then for every $n \in \mathbb{N}$, we have   
$$
{\rm ord}_p(a_n) =
\begin{cases}
{\rm ord}_p(a_{k(p)}), & \text{if } k(p) \mid n \\
0, & \text{else. }
\end{cases}
$$
\end{lemma}

\begin{proof} Write $k = k(p).$ Since $p$ divides $a_k$, $p$ does not divide $b$, so ord$_p(a_n) = $ ord$_p(f^n(0))$ for all $n \in \mathbb{N}$.  For $n \geq 1$, let $g_n(z)$ be the polynomial defined by $f^n(z) = zg_n(z) + f^n(0)$.  Note $g_n(0) = (f^n)'(0) = 0$ for all $n$.  \\

Suppose that $k \mid n$; write $n = mk$ with $m \in \mathbb{N}$.  We have:

$$f^{mk}(0) = f^{(m-1)k + k}(0) = f^{(m-1)k}(0) \cdot g_k(f^{(m-1)k}(0)) + f^k(0);$$
assuming inductively that ord$_p(f^{(m-1)k}(0)) = $ ord$_p(a_k)$, $g_k(0) = 0$ implies that ord$_p(f^n(0)) = $ ord$_p(f^k(0)) =$ ord$_p(a_k)$.  \\

Now suppose that $k$ does not divide $n$.  Write $n = qk+r$ with $0 < r < k,$ noting that by definition of $k$, ord$_p(f^r(0)) = 0$.  Then we have

$$f^n(0) = f^{qk+r}(0) = f^{qk}(0) g_r(f^{qk}(0)) + f^r(0);$$
since ord$_p(f^{qk}(0)) > 0,$ ord$_p(f^n(0)) = 0$ as desired.
\end{proof}

\begin{corollary} Suppose that $n \in \mathbb{N}$ such that $a_n$ has no primitive prime divisor, i.e. $n \in \mathcal{Z}(f,0)$.  Then 
$$a_n \mid \prod_q a_{\frac{n}{q}},$$
 where the product is taken over all distinct primes $q$ which divide $n$.
\end{corollary}

\begin{proof}  Suppose $p$ is a prime dividing $a_n$.  Let $k$ be minimal such that $p \mid a_{k}.$  Since $p$ is not a primitive divisor, $k < n,$ and by the lemma, $k$ divides $n$.  Thus $k$ divides $\frac{n}{q}$ for some prime $q$ dividing $n$; but by the lemma, ord$_p(a_n) = $ ord$_p(a_{k}) = $ ord$_p(a_{\frac{n}{q}})$.  Taking the product over all $p$ yields the corollary.
\end{proof}

Taking absolute values and logarithms, we immediately have the following inequality which will provide the starting point of all effective computations:

\begin{corollary} Suppose that $a_n$ has no primitive prime divisor.  Then 

\begin{equation} \label{eqn1} \log |a_n| \leq \sum_q \log |a_{\frac{n}{q}}|;
\end{equation}
consequently, 

\begin{equation} \label{eqn2} \log{|f^n(0)|} + d^{n-1} \log{b} \leq \sum_q (\log{|f^{\frac{n}{q}}(0)|} + d^{\frac{n}{q} - 1} \log{b}).
\end{equation}
where the sum is taken (without multiplicity) over the primes $q$ which divide $n$.  

\end{corollary}

Because we seek to derive a contradiction from the above inequalities, it is convenient to treat the $n=2$ case separately.  By definition, 

$$f^2(0) = \frac{a^d}{b^d} + \frac{a}{b} = \frac{a^d + ab^{d-1}}{b^d} = \frac{a(a^{d-1} + b^{d-1})}{b^d}.$$
Thus $2 \in \mathcal{Z}(f,0)$ if and only if $a^{d-1} + b^{d-1} = \pm 1;$ since $b \geq 2$ and $a \ne 0$, this holds if and only if $d = 2$ and $a = - (b \pm 1).$  Therefore we conclude the following:

\begin{proposition} \label{N2} For any $c = \frac{a}{b} \in \mathbb{Q} \setminus \mathbb{Z}$, $2 \in \mathcal{Z}(f,0)$ if and only if $d = 2$ and $a = - (b \pm 1).$
\end{proposition}

\begin{remark} Note then that $c > 0$ implies $2 \notin \mathcal{Z}(f,0)$.  
\end{remark}

\ssection{Bounding $\#\mathcal{Z}(f, 0)$} \label{Eff1}

In this section we prove Theorem \ref{DAEff}, which provides a uniform bound on the size of the Zsigmondy set for any value of $c \in \mathbb{Q}.$  Since the result of Theorem \ref{DAEff} is superseded by that of Theorem \ref{NREff} if applicable, we will assume that the hypotheses of Theorem \ref{NREff} do not apply; namely, that we have $d$ even and $c \in (-2^{\frac{1}{d-1}}, -1)$ (see Section \ref{tight} for the proof of Theorem \ref{NREff}).  Our goal is to use inequality (\ref{eqn2}), for which we require both upper and lower bounds on $|f^n(0)|$.  Our assumption on $c$ yields a strong upper bound via induction: 

\begin{lemma} \label{bound} Suppose $d$ is even and $-2^{\frac{1}{d-1}} < c < -1$.  Then we have 

$$|f^n(0)| \leq |c|$$
for all $n \geq 1$.  

\end{lemma}

Thus supposing that $n \in \mathcal{Z}(f,0)$, inequality (\ref{eqn2}) implies

$$\log{|f^n(0)|} + d^{n-1} \log{b} \leq \omega(n) \log{|c|} + \sum_q d^{\frac{n}{q}-1} \log{b},$$
where we define $\omega(n)$ to be the number of distinct prime factors of $n$.   Since a version of the above inequality will be used many times in this paper, it is worth noting that we have the following coarse bounds:
$$s_d(n) \leq d^{\frac{n}{2}} \log_2(n), \ \omega(n) \leq \log_2(n),$$
since each prime factor of $n$ is bounded below by 2. \\

To use this inequality to bound $n$, we require a lower bound on $|f^n(0)|$ that is reasonably better than $b^{-d^{n-1}}$.  

\ssubsection{$x^d+c$ irreducible over $\mathbb{Q}$}
{\bf Assumption.} Throughout this subsection, we assume that $d$ is even, $c \in \mathbb{Q}\cap(-2^{\frac{1}{d-1}}, -1)$, and for all $m \mid d$, $m > 1$, $c$ is not an $m$th power of a rational number. \\

Under this assumption, we achieve the following bound on the recurrence of the critical point:

\begin{sstheorem} \label{DABound} For each even $d \geq 2,$ there exist positive integers $1 \leq N_d \leq 6$ and $1 \leq m_d \leq 6$ such that there are at most $N_d$ values of $n \in \mathbb{N}$ satisfying both 

$$n \geq 2m_d + 6$$
and
$$|f^n(0)| \leq (b^{d^{n-2}})^{-d(1-d^{-m_d})}.$$

Further, for $d \geq 6,$ the result holds with $m_d = 1$ and $N_d = 2$, and for $d = 4$, the result holds with $m_d = 2$ and $N_d = 3$. 
\end{sstheorem}

For those $c$ satisfying the assumptions, Theorem \ref{DAEff} is an immediate consequence of Theorem \ref{DABound}:

\begin{proof} [Proof of Theorem \ref{DAEff}] If $n \in \mathcal{Z}(f,0)$, Lemma \ref{bound} and inequality (\ref{eqn2}) together imply

$$\log{|f^n(0)|} + d^{n-1} \log{b} \leq \frac{1}{d} s_d(n) \log{b} + \omega(n) \log{|c|},$$
where (as above)
$$s_d(n) := \sum_q d^{\frac{n}{q}}$$
is a sum over primes $q$ dividing $n$, and $\omega(n)$ is the number of distinct primes dividing $n$.   \\

Thus applying Theorem \ref{DABound}, for all but at most $N_d$ values of $n$ with $n \geq 2m_d+6$, we have 
$$d^{n-m_d-1} - \frac{1}{d} s_d(n) \leq \frac{\omega(n) \log{|c|}}{\log{b}}.$$
Since $\frac{a}{b} = c \in (-2, -1)$, we have $|c| < b$; also we have $\frac{1}{d} s_d(n) < d^{\frac{n}{2}}$.  Thus
$$d^{n-m_d-1} - d^{\frac{n}{2}} \leq \omega(n).$$
By assumption, $n - m_d - 1 \geq \frac{n}{2} + 2$, and so 
$$d^{\frac{n}{2}} \leq \omega(n);$$
since $\omega(n) \leq \log_2(n)$, this is false for all $d \geq 2, n \geq 2$.   \\

Therefore the size of the Zsigmondy set satisfies

$$\#\mathcal{Z}(f,0) \leq 2m_d+6 -1 + N_d \leq 23$$
for all values of $d \geq 2$, with improved bound for $d = 4$ of 
$$\#\mathcal{Z}(f,0) \leq 2m_d+6 -1 + N_d = 12,$$
and for $d \geq 6$ we have
$$\#\mathcal{Z}(f,0) \leq 2m_d+6 -1 + N_d = 9.$$
\end{proof}

The remainder of this section will be devoted to the proof of Theorem \ref{DABound}, which relies on the proof of Mahler's quantitative result \cite{Mahler:OTCF} on restricted rational approximation of real algebraic numbers.   Examining the proof of Theorem 3 of \cite{Mahler:OTCF}, we extract the following quantitative statement bounding the good rational approximates of real algebraic numbers:

\begin{sstheorem} \label{Mahler} [Mahler] Let $S$ be a finite set of primes, $\zeta$ a real algebraic number of degree $d \geq 2$, and $\mu > \sqrt{d}$.  Let $R$ be the maximal absolute value of the coefficients of the minimal integral polynomial of $\zeta$. Suppose $\epsilon > 0$ is sufficiently small so that 
$$\kappa := \left( \sqrt{\frac{1-2\epsilon}{d}} - 2\sqrt{\epsilon} \right) \mu - (1+\epsilon)^2 > 0.$$
Then there do not exist rational $S$-integers $\frac{p_1}{q_1}, \frac{p_2}{q_2}$ satisfying 
\begin{equation} \label{GoodApprox} \left| \frac{p_i}{q_i} - \zeta \right| < q_i^{-\mu},
\end{equation} 
which also satisfy
\begin{itemize}
\item $q_1^{\kappa} \geq (16R)^{\frac{4}{\epsilon}},$
\item $q_2 \geq q_1^{\frac{5d^2}{2\epsilon}}.$
\end{itemize}
\end{sstheorem}

To apply this theorem to our setting, let $\zeta$ be the positive $d$th root of $|c|$, and $\mu = d(1-d^{-m})$, with $m$ to be chosen later.  Since $\zeta > 1,$ we have

$$|\frac{|a_{n-1}|}{b^{d^{n-2}}} - \zeta| < |f^{n-1}(0)^d - |c|| = |f^n(0)|,$$
so if 
$$|f^n(0)| \leq (b^{d^{n-2}})^{-d(1-d^{-m})},$$
then $|f^{n-1}(0)|$ is a good approximate of $\zeta$ in the sense of inequality (\ref{GoodApprox}). \\

So we will apply Mahler's theorem to the iterates $f^{n-1}(0)$; to do so, we rewrite the last three conditions of Theorem \ref{Mahler} in our setting.  Suppose that $|f^{n_1-1}(0)|$ and $|f^{n_2-1}(0)|$ are both good approximates to $\zeta$; i.e., satisfy inequality (\ref{GoodApprox}).  Since the denominator of $|f^{n_i-1}(0)|$ is $q_i = b^{d^{n_i-2}},$ we have

$$q_1^{\kappa} \geq (16R)^{\frac{4}{\epsilon}} \Leftrightarrow d^{n_1-2} \log{b} \geq \frac{4}{\kappa\epsilon} \log{16R}.$$

Since $|c| \in (1, 2)$, we have $R = |a| < 2b$; also we have $b \geq 2,$ so 

$$n_1 \geq \log_d(\frac{24}{\kappa \epsilon}) + 2 \Rightarrow (b^{d^{n_1-2}})^{\kappa} \geq (b^6)^{\frac{4}{\epsilon}} \Rightarrow d^{n_1-2} \geq \frac{24}{\kappa \epsilon} \Rightarrow q_1^{\kappa} \geq (16R)^{\frac{4}{\epsilon}}.$$

Similarly, we have 

$$n_2 \geq n_1 + \log_d(\frac{5d^2}{2\epsilon})  \Rightarrow q_2 \geq q_1^{\frac{5d^2}{2\epsilon}}.$$

Therefore we have shown that Theorem \ref{Mahler} implies the following:

\begin{ssproposition} \label{Approx} Suppose that $|f^{n_1-1}(0)|$ and $|f^{n_2-1}(0)|$ satisfy inequality (\ref{GoodApprox}), with $n_1 \geq \log_d(\frac{24}{\kappa \epsilon}) + 2.$ Then we have 

$$n_2 < n_1 + \log_d(\frac{5d^2}{2\epsilon}).$$
\end{ssproposition}

\begin{proof} [Proof of Theorem \ref{DABound}] According to Proposition \ref{Approx}, in order to prove Theorem \ref{DABound}, we must show that we can choose $\epsilon$ and $\mu = d(1-d^{-m_d})$ such that $\kappa > 0$, and 
\begin{itemize}
\item $1 \leq m_d \leq 6$,
\item $2m_d+6 \geq \log_d(\frac{24}{\kappa \epsilon}) +2$, and 
\item $\log_d(\frac{5d^2}{2\epsilon}) \leq 6.$
\end{itemize}
\begin{ssremark} \label{NonConsecutive} In fact, we can weaken this last inequality; since $|f^n(0)| < \frac{1}{2} \Rightarrow |f^{n+1}(0)| > \frac{1}{2}$, we cannot have consecutive good approximates, and so we have $N_d \leq \frac{1}{2}\log_d(\frac{5d^2}{2\epsilon})$. 
\end{ssremark}

Suppose $d \geq 6$.  Let $m_d = 1$ and $\epsilon = \frac{1}{d^3}$ (note $\mu = d-1$).  Then one can compute that $\kappa > \frac{24}{d^3} > 0$, and therefore 

$$\log_d(\frac{24}{\kappa \epsilon}) < 6.$$
Therefore
$$2m_d+6 = 8 \geq  \log_d(\frac{24}{\kappa \epsilon}) +2.$$

By choice of $\epsilon,$ we have

$$\log_d(\frac{5d^2}{2\epsilon}) = 5 + \log_d(\frac{5}{2}) < 6,$$ 

and by Remark \ref{NonConsecutive}, we conclude that $N_d \leq \frac{1}{2}\log_d(\frac{5d^2}{2\epsilon}) < 3$.  \\

For $d=2$, we simply note that the smallest $m_2$ and $N_2$ that can be achieved are found when $\epsilon = .004$ and $m_2 = 6$.  In this case we have
$$2m_2 + 6 = 18 \geq \log_2(\frac{24}{\kappa \epsilon}) + 2,$$
and
$$\log_2(\frac{5d^2}{2\epsilon}) = \log_2(15000) < 14,$$
so we can take $N_2 = 6.$ \\

Similarly for $d=4,$ we achieve optimal values at $m_4 =2$ and $\epsilon = \frac{1}{128}$.  In this case we have 
$$2m_4 + 6 = 10 \geq \log_4(\frac{24}{\kappa \epsilon})+2,$$
and
$$\log_2(\frac{5d^2}{2\epsilon}) = \frac{11}{2} + \log_4(\frac{5}{2}) < 7,$$
so we can take $N_4 = 3.$

\end{proof}

\begin{ssremark} In Mahler's proof, the goal was to achieve the result for the most general case.  In our situation, the simplicity of the minimal polynomial yields slightly stronger results if we tighten the Diophantine approximation by hand.  In particular, one can show that there is at most one $n \in \mathcal{Z}(f,0)$ with $n \geq 7$.  However, the proof is a lengthy and relatively unenlightening computation, so we choose to use Mahler's result, at the expense of the bound on the size of $\mathcal{Z}(f,0)$; see \cite{Thesis} for this computation.
\end{ssremark}

\ssubsection{$x^d + c$ reducible over $\mathbb{Q}$}

In the case when $|c|^{\frac{1}{d}}$ has degree less than $d$ over $\mathbb{Q}$ we have a stronger result: 

\begin{ssproposition} Suppose that $d$ is even, and $c = \frac{a}{b} \in (-2^{\frac{1}{d-1}}, -1)$ such that there exists $m \mid d$, $m \ne 1$, and positive integers $k, l$ with $a = - k^m, b = l^m$.  Then $\mathcal{Z}(f,0) = \emptyset.$
\end{ssproposition}

In order to prove the proposition, we find a lower bound for $|f^n(0)|$:

\begin{sslemma} Suppose $d$ and $c$ are as above.  Then we have 

$$|f^n(0)| \geq \frac{1}{db^{\frac{1}{2}d^{n-1}}}$$
for all $n \geq 2.$
\end{sslemma}

\begin{proof} By assumption, we have 

\begin{eqnarray*}|f^n(0)| &=& |f^{n-1}(0)^d - \left(\frac{k}{l}\right)^m| \\
&=& | |f^{n-1}(0)|^{\frac{d}{m}} - \frac{k}{l}| \cdot |\left(|f^{n-1}(0)|^{\frac{d}{m}}\right)^{m-1} + \left(|f^{n-1}(0)|^{\frac{d}{m}}\right)^{m-2} \cdot \left( \frac{k}{l} \right)+ \cdots + \left( \frac{k}{l} \right)^{m-1}| \\
&\geq& | |f^{n-1}(0)|^{\frac{d}{m}} - \frac{k}{l}|,
\end{eqnarray*}
since the right-hand factor is a sum of positive numbers, one of which is $\left( \frac{k}{l} \right)^{m-1},$ which is $>1$ by assumption.  \\

Write $\beta$ for the positive $\frac{d}{m}$th root of $\frac{k}{l}$ - for notational convenience we will set $r = \frac{d}{m}$, so that $\beta^r = \frac{k}{l}$.  From the above, we have

\begin{eqnarray*} |f^n(0)| &\geq&  | |f^{n-1}(0)|^{\frac{d}{m}} - \frac{k}{l}| \\
&=& | |f^{n-1}(0)| - \beta| \cdot | |f^{n-1}(0)|^{\frac{d}{m}-1} + |f^{n-1}(0)|^{\frac{d}{m}-2} \cdot \beta + \cdots + \beta^{\frac{d}{m} - 1} | \\
&>& | |f^{n-1}(0)| - \beta|,
\end{eqnarray*}
since $\beta > 1$.  But we know that $|f^{n-1}(0)|$ is a rational number whose denominator is $b^{d^{n-2}}$ and therefore a power of $l$.  Therefore we have 

\begin{eqnarray*} \frac{1}{b^{r \cdot d^{n-2}}} &\leq& |(\frac{|a_{n-1}|}{b^{d^{n-2}}})^r - \frac{k}{l}| \\
&=& ||f^{n-1}(0) - \beta| \cdot | |f^{n-1}(0)|^{r-1} + ... + \beta^{r-1}| \\ &\leq& ||f^{n-1}(0)| - \beta| \cdot r \max \{ |f^{n-1}(0)|, \beta \}, \\
\end{eqnarray*}
noting that the first inequality is valid because the right-hand term divides $f^n(0)$ and thus cannot be 0, since 0 is not periodic.  By Lemma \ref{bound}, we then have
\begin{eqnarray*} \frac{1}{b^{r \cdot d^{n-2}}} &\leq& | |f^{n-1}(0)| - \beta | \cdot \frac{d}{2} \cdot |c| \\
&\leq&  | |f^{n-1}(0)| - \beta | \cdot d;
\end{eqnarray*}
Since $r = \frac{d}{m} \leq \frac{d}{2},$ we conclude that 

$$|f^n(0)| > | |f^{n-1}(0)| - \beta | \geq \frac{1}{d \cdot b^{r \cdot d^{n-2}}} \geq  \frac{1}{d \cdot b^{\frac{1}{2} d^{n-1}}},$$
as desired. 
\end{proof}

Having achieved a lower bound for $|f^n(0)|$, we can now prove the proposition.

\begin{proof} Suppose $n \geq 3$ with $n \in \mathcal{Z}(f,0)$, so that 

$$\log{|f^n(0)|} + d^{n-1} \log{b} \leq \omega(n) \log{|c|} + \sum_q d^{\frac{n}{q}-1} \log{b}.$$
By the lemma, we then have

$$d \log \left( \frac{1}{d \cdot b^{\frac{1}{2} d^{n-1}}} \right)+ d^n \log{b} < d \omega(n) \log{|c|} + s_d(n) \log{b},$$
and so 

$$- d \log{d} + \frac{1}{2} d^n \log{b} < d \omega(n) \log{|c|} + s_d(n) \log{b}.$$
Since $|c| < 2,$

$$\frac{1}{2} d^n - s_d(n) < \frac{d \omega(n) \log{2} + d \log{d}}{\log{b}}.$$
Since $c$ is an $m$th power of a rational number, $m > 1,$ we have $b \geq 9$, so 

$$\frac{1}{2} d^n - s_d(n) < \frac{d}{3} \omega(n) + \frac{1}{2} d \log{d},$$
and so 

$$\frac{1}{2} d^{n-1} - \frac{1}{2}\log{d} - \frac{1}{d} s_d(n) \leq \frac{1}{3} \omega(n).$$
Utilizing the bounds $s_d(n) \leq d^{\frac{n}{2}} \log_2(n)$ and $\omega(n) \leq \log_2(n),$ we see that this is false for all $d \geq 2$ and all $n \geq 3$.   \\

Since from Proposition \ref{N2} we know that $2 \in \mathcal{Z}(f,0)$ only if $d =2$ and $a = - (b \pm 1)$, our assumption that $c$ is an $m$th power guarantees that $2 \notin \mathcal{Z}(f,0)$, and the proposition is proved.

\end{proof}

\ssection{Existence of an Effective Bound $M(c)$} \label{Eff2}

Theorem \ref{NREff} (see Section \ref{tight}) guarantees a maximal element of 2 in the Zsigmondy set except in the possibly recurrent case of $d$ even and $c \in (-2^{\frac{1}{d-1}}, -1).$  However, it is possible regardless of choice of $c$ to use effective Diophantine approximation to bound the maximal element of the Zsigmondy set.  In this section, we prove Theorem \ref{EffExists} using an improvement of Schinzel's result \cite{Schinzel:OTTO} on approximation of quadratic irrationals due to Bennett and Bugeaud \cite{BennettBugeaud:ERFR}:

\begin{theorem} \label{BB} [Theorem 1.2 of \cite{BennettBugeaud:ERFR}] Let $||x||$ denote the distance from $x$ to the nearest integer.
For every integer $b \geq 2$ and every quadratic real number $\xi$, there exist positive effectively computable constants $\epsilon(\xi, b)$ and $\tau(\xi, b)$ such that for all $n \geq 1,$

$$||b^n \xi|| > \frac{\epsilon(\xi, b)}{b^{-(1-\tau(\xi, b))n}}.$$
\end{theorem}

\begin{proof} [Proof of Theorem \ref{EffExists}] We may by Theorem \ref{NREff} simplify our argument by assuming $c = \frac{a}{b} \in (-2^{\frac{1}{d-1}}, -1)$ and $d$ even, so write $c = - \xi^2$, choosing positive square root $\xi$.  For notational convenience, we denote the constants of Theorem \ref{BB} by $\epsilon$ and $\tau$ respectively, so that for all $n \geq 1$, we have 

$$||b^n \xi|| > \frac{\epsilon}{b^{-(1-\tau))n}}.$$

Then for any $n \geq 1,$ we have the following lower bound for $|f^n(0)|$:

\begin{eqnarray*} |f^n(0)| &>& \left| \left( \frac{|a_{n-1}|}{b^{d^{n-2}}} \right)^{\frac{d}{2}}- \xi \right| \\
&=& b^{-\frac{1}{2}d^{n-1}} \left| |a_{n-1}| - \xi b^{-\frac{1}{2}d^{n-1}} \right| \\
&\geq& b^{-\frac{1}{2}d^{n-1}} ||\xi b^{\frac{1}{2}d^{n-1}}|| \\
&\geq& b^{-\frac{1}{2}d^{n-1}} \frac{\epsilon}{b^{\frac{1}{2}(1-\tau)d^{n-1}}} \\
&=& \frac{\epsilon}{b^{d^{n-1} - \frac{\tau}{2} d^{n-1}}}
\end{eqnarray*}

Suppose now that $n \notin \mathcal{Z}(f,0)$, so that by inequality (\ref{eqn1}) and Lemma \ref{bound}, we have 

$$d \log{|f^n(0)|} + d^n \log{b} \leq s_d(n) \log{b} + d \omega(n) \log{|c|},$$

where as before $s_d(n) := \sum_q d^{\frac{n}{q}}$  and $\omega(n) := \sum_q 1$ are sums over the distinct prime factors $q$ of $n$.  Our lower bound on $|f^n(0)|$ then implies

$$\log{\epsilon} + \frac{\tau}{2} d^{n-1} \log{b} \leq \frac{1}{d} s_d(n) \log{b} + \omega(n) \log{|c|},$$
so 
$$(\frac{\tau}{2} d^{n-1} -\frac{1}{d} s_d(n)) \log{b} \leq \log{\frac{1}{\epsilon}} + \omega(n) \log{|c|}.$$
Since $\tau$ is a constant, the left-hand side growth will be exponential in $n$ for $n$ sufficiently large, while the right-hand side is $\mathcal{O}(\log{n})$.  Thus for  $n \geq M(c)$ for some sufficiently large $M(c)$, we have a contradiction.  Further, since $\tau$ and $\epsilon$ are effectively computable, $M(c)$ is as well.
\end{proof} 

\begin{remark} The existence of the constant $\tau$ is a consequence of an effective linear forms in logarithms bound and is not computed in \cite{BennettBugeaud:ERFR}, but has a complicated dependence on $\xi$.  Working through the proof of Bennett and Bugeaud's theorem, $\tau$ can be seen to be generally too small for a useful effective bound on the maximal element of the Zsigmondy set; in fact, it is on the order of the reciprocal of the logarithm of the fundamental unit of $\mathbb{Q}(\sqrt{c})$, so $M(c)$ is comparable to the logarithm of the regulator of $\mathbb{Q}(\sqrt{c})$ plus a constant which is large for dynamical purposes.  For example, for $f(z) = z^2 - \frac{3}{2}$, the process gives a value of $M(c)$ close to 80, and computationally checking primitive divisors for 80 iterates is an infeasible task.
\end{remark}

\ssection{The Non-recurrent Case} \label{tight}

In this section we demonstrate that often $M(c)$ is quite small; in fact, if $c$ is chosen so that the critical orbit escapes to infinity, or simply avoids coming back too close to 0, we have $M(c) = 2$.  To that end we prove Theorem \ref{NREff}.  Recall our assumption that $c \notin \mathbb{Z}$.  In order to utilize inequalities \ref{eqn1} and \ref{eqn2}, we connect the elements of the sequence $\{ a_n \}$ to the corresponding Weil heights $ h(f^n(0))$, or find bounds on the modulus of the critical orbit, respectively.  In the case when $|c| > 2^{\frac{d}{d-1}}$, we can successfully use the former approach.

\begin{lemma}  Suppose that $c$ satisfies $|c| > 2^{\frac{d}{d-1}}$.  Then $|f^n(0)| > |c| > 1$ for all $n \geq 2$.
\end{lemma}

\begin{proof} Since 

$$|f^n(0)| = |f^{n-1}(0)^d + c| = |c| \cdot |\frac{f^{n-1}(0)}{c} \cdot f^{n-1}(0)^{d-1} + 1|,$$
the lemma is immediate by induction.
\end{proof}

Denote by $h$ the standard logarithmic Weil height $h(P)$ on $\mathbb{P}^1(\mathbb{Q})$.  We will abuse notation and use $h$ as a height on $\mathbb{Q}$ as well, in which case we have, for $\frac{r}{s}$ in lowest terms,

$$h(\frac{r}{s}) = \log(\max \{ |r|, |s| \}).$$
By the lemma above, when $|c| > 2^{\frac{d}{d-1}},$ the inequality (\ref{eqn1}) becomes the following:

$$h(a_n) \leq \sum_q h(a_{\frac{n}{q}}).$$
Define the usual dynamical canonical height 

$$\hat{h}_f(P) = \lim_{n \rightarrow \infty} \frac{f^n(P)}{d^n},$$ 
and recall that for all $P$ and all $n \in \mathbb{N}$, 

$$\hat{h}_f(f^n(P)) = d^n \hat{h}_f(P).$$
Further, there exists a constant $C$ such that for all $\alpha \in \mathbb{Q}$, 

\begin{equation} \label{HtC} |h(\alpha) - \hat{h}_f(\alpha)| < C.
\end{equation}
Therefore our inequality becomes

\begin{equation} \label{eqn3} \hat{h}_f(a_n) - C \leq \sum_q (\hat{h}_f(a_{\frac{n}{q}}) + C).
\end{equation}
To achieve an effective result, we make the constant $C$ explicit in the following lemma.  \\

\begin{lemma} \label{Const} Let $f(z) = z^d + c$ be as above.  Then we can take the constant $C$ of inequality (\ref{HtC}) to be $h(c)+\log(2)$.  
\end{lemma}

\begin{proof} We use the methods of Theorems 3.11 and 3.20 of \cite{Silverman:TAOD}.  Consider $f$ as a morphism  $[\phi_z : \phi_w]$ on $\mathbb{P}^1$ given by $[z : w] \mapsto [z^d+cw^d: w^d]$.  Let $h$ denote the logarithmic Weil height as above, and for each place $v$ of $\mathbb{Q}$, $h_v$ the local height at $v$.  Since
$$|z^d+c|_v \leq \delta_v \max \{ |z|_v^d, |c|_v \},$$
where $\delta_v = 1$ for $v$ non-archimedean and $\delta_v = 2$ for the archimedean place, we have 
$$h_v(\phi(P)) \leq \log{\delta_v} + d h_v(P) + h_v(c).$$
Similarly, we have
$$|z^d|_v \leq |z^d + c - c|_v \leq \delta_v \max \{ |z^d+c|_v, |c|_v \},$$
so 
$$d h_v(P) \leq \log{\delta_v} +  h_v(\phi(P)) + h_v(c).$$
Combining these estimates and taking the sum over all places of $\mathbb{Q}$, we see that
$$-\log{2} - h(c) + h(\phi(P)) \leq dh(P) \leq \log{2} + h(\phi(P)) + h(c),$$
and so 
$$|h(\phi(P)) - dh(P)| \leq h(c) + \log(2).$$

Taking a telescoping sum, we see that

$$|\hat{h}_f(P) - h(P)| \leq \frac{h(c) + \log(2)}{d-1} \leq h(c) + \log(2),$$
as desired.
\end{proof}	

We can now prove an effective Zsigmondy result: \\

\begin{proposition} \label{bigc} Suppose $|c| > 2^{\frac{d}{d-1}}$.  Then $\mathcal{Z}(f,0) = \emptyset$.  
\end{proposition}

\begin{proof} First note that by Proposition \ref{N2} and the assumption $|c| > 2^{\frac{d}{d-1}} > 2$, it suffices to prove that $n \notin \mathcal{Z}(f,0)$ for all $n \geq 3$. \\

Let $C$ be as given in inequality (\ref{HtC}). By inequality (\ref{eqn3}), if $a_n$ fails to have a primitive prime divisor, we have 

$$\hat{h}_f(f^n(0)) - C \leq \omega(n)C+ \sum_q \hat{h}_f(f^{\frac{n}{q}}(0)),$$
where $\omega(n)$ denotes the number of distinct prime divisors of $n$.  Therefore we have

$$d^n \hat{h}_f(0) \leq (\omega(n)+1)C + \hat{h}_f(0) \sum_q d^{\frac{n}{q}}.$$

To simplify notation, write $s_d(n) = \sum d^{\frac{n}{q}},$ with $q$ taken over distinct primes of $n$.  Since 0 is not preperiodic, $\hat{h}_f(0) >0,$ and the inequality above can be written 

$$\frac{d^n - s_d(n)}{\omega(n)+1} \leq \frac{C}{\hat{h}_f(0)};$$
by Lemma \ref{Const}, we conclude that

$$\frac{d^n - s_d(n)}{\omega(n)+1} \leq \frac{h(c) + \log{2}}{\hat{h}_f(0)}.$$

We now use a remark following Lemma 6 of \cite{Ingram:LBOT} to get a lower bound for $\hat{h}_f(0)$:

\begin{lemma} [Ingram] Suppose $|c| > 2^{\frac{d}{d-1}}$, and $f(z) = z^d+c$.  Then 

$$\hat{h}_f(c) \geq \frac{1}{d} h(c).$$
\end{lemma}
Consequently, we have
$$\hat{h}_f(0) \geq \frac{1}{d^2} h(c).$$
Thus if $a_n$ has no primitive prime divisor, $n$ must satisfy 

$$\frac{d^n - s_d(n)}{\omega(n)+1} \leq d^2 \frac{h(c)+\log{2}}{h(c)} = d^2 (1 + \frac{\log{2}}{h(c)}) < 1.5 d^2,$$
where the right-hand inequality holds because $b \geq 2$ and $|c| > 2^{\frac{d}{d-1}} > 2$ together imply that $h(c) = \log|a| \geq \log{4}$. \\

Since $d^n - s_d(n)$ grows very quickly with $n$, this gives a strong restriction on $n$; in fact, one can use the bounds $s_d(n) \leq d^{\frac{n}{2}} \log_2(n)$ and $\omega(n) \leq \log_2(n)$ to see that 
$$\frac{d^n - s_d(n)}{(\omega(n) + 1)} > 1.8d^2$$
if $d \geq 4, n \geq 3$, or $d \geq 3, n \geq 4,$ or $d \geq 2, n \geq 5$. \\

Thus the only cases that remain are $d = 3$ and $n =3,$ or $d=2$ and $n = 3,4$, which we check by hand.  \\

If $d=3$ and $n=3$, we compute

$$f^3(0) = \frac{a}{b^{9}} (a^2(a^2+b^2)^3 + b^{8}).$$  
Since $a$ and $b$ are coprime, the term $(a^2(a^2+b^2)^3 + b^{26})$ can have no common divisors with $a$; but since it is a sum of positive integers and $b \geq 2$, $(a^2(a^2+b^2)^3 + b^{26}) \geq 2$ and so is divisible by some prime.  \\

Therefore $a_3$ has a primitive prime divisor for $d =  2$ or $3$.  \\

Finally we turn to the case when $d = 2$ and $n = 4$.  If $d=2$ and $4 \in \mathcal{Z}(f,0),$ then we have 

$$\frac{16 - 4}{4} \leq (1 + \frac{\log{2}}{h(c)}) (\omega(4) + 1),$$
and so 

$$\frac{1}{2} \leq \frac{\log{2}}{h(c)},$$
and so $a \leq 9.$   But by assumption, we have $\frac{a}{b} > 2^{\frac{2}{1}} = 4,$ so the only possibility is $a = \pm 9$ and $b = 2$.  One can check by hand that for these values of $c$, $a_3$ has a primitive prime divisor, and the proposition is proved.
\end{proof}

In the remainder of this section, we cannot necessarily utilize height functions, but non-recurrence of the critical orbit will provide upper and lower bounds on $|f^n(0)|$ for all $n$, which can be used in conjuction with inequality (\ref{eqn2}). \\

We have straightforward bounds when $c$ is positive; the proof of the following lemma is an easy induction: \\

\begin{lemma} Suppose $c > 0$.  Write $C(n) = \max \{ c, c^{d^{n-1}} \}.$ Then for all $n \geq 1$, we have

$$C(n) \leq f^n(0) \leq 2^{\frac{d^{n-1}-1}{d-1}} C(n).$$

\end{lemma}

\begin{proposition} \label{dodd} Suppose $c > 0$, or $c < 0$ and $d$ is odd.  Then $\mathcal{Z}(f,0) = \emptyset$.
\end{proposition}

\begin{proof} First note that it is sufficient to prove the proposition for $c >0$, since if $c<0$ and $d$ is odd, we may replace $c$ with $-c$ and the forward orbit of 0 will be unchanged, modulo sign.  Therefore we assume that $c > 0$ (and thus the forward orbit consists of positive numbers).  In light of the remark following Proposition \ref{N2}, we must prove that $n \notin \mathcal{Z}(f,0)$ for all $n \geq 3$ and all $d \geq 2.$ \\

We recall that if $n \in \mathcal{Z}(f,0)$, then we have

$$\log{f^n(0)} + d^{n-1} \log{b} \leq \sum_q (\log{f^{\frac{n}{q}}(0) }+ d^{\frac{n}{q}-1} \log{b}),$$
with the sum over distinct primes $q$ dividing $n$.  Multiplying by $d$ and applying the preceding lemma, we have:

$$ d \log{C(n)} + d^n \log{b} \leq \sum_q \left[ \frac{d^{\frac{n}{q}}-d}{d-1} \log{2} + d\log{C(\frac{n}{q})} + d^{\frac{n}{q}} \log{b} \right];$$
rearranging, we have

$$d \left[ \log{C(n)} - \sum_q \log{C(\frac{n}{q})} \right] + \left[ d^n - s_d(n) \right] \log{b} \leq  \frac{1}{d-1} s_d(n) \log{2}.$$

Checking by cases, we see that the left-most term is always non-negative, and therefore we have the inequality

$$\left[ d^n - s_d(n) \right] \log{b} \leq \frac{1}{d-1} s_d(n) \log{2}.$$

By assumption, $c$ is non-integral and so $b \geq 2$, and therefore

$$\left[ d^n - s_d(n) \right] \log{2} \leq \frac{1}{d-1} s_d(n) \log{2},$$
and so 

$$d^n - \frac{d}{d-1} s_d(n)  \leq 0,$$
which is impossible for any $d \geq 2$, $n \geq 3$.  \\

\end{proof}

Next we consider the situation when $-1 < c < 0$ and $d$ is even:

\begin{proposition} \label{more1} Suppose $-1<c<0$ and $d$ is even.  Then $\mathcal{Z}(f,0) = \emptyset$, unless $d = 2$ and $a + b = 1,$ in which case $\mathcal{Z}(f,0) = \{ 2 \}.$
\end{proposition}

\begin{proof} By Proposition \ref{N2}, we must prove $n \notin \mathcal{Z}(f,0)$ for all $n \geq 3.$  We utilize the following bounds, which by assumption on $c$ and $d$ hold for all $n \geq 0$:

$$|c| (1-|c|^{d-1}) \leq |f^n(c)| \leq |c|.$$

Together  inequality (\ref{eqn2}) and these bounds imply that we have $n \in \mathcal{Z}(f,0)$ only if 

$$\log(|c|(1-|c|^{d-1})) + d^{n-1} \log{b} \leq \omega(n)\log|c| + \log{b} \sum_q d^{\frac{n}{q}-1}.$$
Multiplying by $d$ and rearranging, we have

\begin{eqnarray*} (d^n - s_d(n)) \log{b} &\leq& d(\omega(n) -1) \log{|c|} - d \log(1 - |c|^{d-1}) \\
&\leq& - d \log(1 - |c|^{d-1}) \\
&=& d(d-1) \log{b} - d \log(b^{d-1} - |a|^{d-1}) \\
&\leq& d(d-1) \log{b}. \\
\end{eqnarray*}
We conclude that 

$$d^n - s_d(n) - d^2 + d \leq 0,$$
which is impossible for all $d \geq 2$, $n \geq 3.$  Thus the proposition is proved. 
\end{proof}

The final non-recurrent case tightens the bound on $|c|$: \\

\begin{proposition} \label{more2} Suppose $2^{\frac{1}{d-1}} < |c| < 2^{\frac{d}{d-1}}$ Then $\mathcal{Z}(f,0) = \emptyset$, unless $d=2$ and $a+b = -1$, in which case $\mathcal{Z}(f,0) = \{ 2 \}$.  
\end{proposition}

\begin{proof} 

If $d$ is odd this follows from Proposition \ref{dodd}, so we assume $d$ is even.  Again it is easy to bound the critical orbit by induction; since $2^{\frac{1}{d-1}} < |c| < 2^{\frac{d}{d-1}}$, we have
$$\log|c| \leq \log|f^n(0)| \leq (3d^{n-1}-1) \log{2}$$
for all $n \in \mathbb{N}.$

Suppose $n \in \mathcal{Z}(f,0)$.  Then combining the above bounds with inequality (\ref{eqn2}), we have 

$$d \log|c| + d^n \log{b} \leq \sum_q ((3d^{\frac{n}{q}} -d)\log{2} + d^{\frac{n}{q}} \log{b}),$$
and so 

$$(d^n - s_d(n)) \log{b} \leq 3s_d(n) \log{2} - \omega(n) d \log{2},$$ 
which is impossible for $b \geq 2, d \geq 2$ and $n \geq 3$. 

\end{proof}

\begin{proof} [Proof of Theorem \ref{NREff}] The Theorem follows immediately from Propositions \ref{bigc}, \ref{dodd}, \ref{more1} and \ref{more2}.
\end{proof}

\ssection{Computing $M(c)$ on the boundary of the Mandelbrot set} \label{Mand}

In Section \ref{tight} we showed that if there were a norm-based obstruction to critical orbit recurrence, then $M(c)$ is quite small.  One can also construct a dynamical obstruction to recurrence and achieve a small $M(c)$ for certain values of $c$, which is outlined below for $d=2$.   In particular, we will prove that $M(c) \leq 3$ for those values of $c$ which are not too close to any $c$ for which 0 is preperiodic; i.e., the centers of hyperbolic components of the Mandelbrot set.  We are indebted to Xavier Buff for suggesting this approach.  \\

Throughout this section, we let $f_c(z) = z^2+c$, considering $c$ as a complex parameter.  For each $n \in \mathbb{N}$, fix $\rho_n > 0$, and define $D(n, \rho_n)$ to be the set of complex parameters $c$ such that 0 lies in an attracting basin of some point $a$ with exact period $n$ and $|(f_c^n)'(a)| \leq \rho_n$.  

\begin{theorem} \label{CA} Define $D(n, \rho_n)$ as above with $\rho_n = \min \{ \frac{1}{4}, \frac{1}{2^{2^{n-2}}} \}$.  Write

$$\mathcal{D} := \mathbb{C} - \bigcup_{n \in \mathbb{N}} D(n, \rho_n).$$
Then for all $c = \frac{a}{b} \in \mathcal{D},$ we can take $M(c) = 3,$ with $3 \in \mathcal{Z}(f,0)$ if and only if $c = -\frac{7}{4}$.  
\end{theorem}

The strategy for proving the above is to conformally conjugate $f^n$ on the immediate basin of the attracting cycle to a Blaschke product on the unit disk.  Then the theorem of de Branges (see Theorem \ref{dB}) and the Maximum-Modulus Principle provide the following uniform bound:

\begin{proposition} Fix $n$ and $0 < \rho < \frac{1}{4}.$  Then for all $c \in \mathbb{C} \setminus (\bigcup_{k \mid n} D(k, \rho)),$ we have 
$$|f_c^n(0)| \geq \rho \cdot \frac{1}{2^{2n+2}}.$$
\end{proposition}
(In particular, if $c$ lies on the boundary of the Mandelbrot set, the lower bound holds for all $n$ and any choice of $\rho < \frac{1}{4}$).

\begin{proof} We will for convenience suppress the dependence of $D(n, \rho)$ on $\rho$. Choose a radius $R$ sufficiently large so that $|f^n_c(0)| \gg 1$ for all $|c| \geq R$, and consider the domain 

$$D_n = \mathbb{D}(0, R) \cap (\mathbb{C} \setminus \bigcup_{k \mid n} D(k)).$$

Since $D_n$ does not contain any point $c$ with critical period dividing $n$, $f^n_c(0)$ is a nonvanishing holomorphic function on $D_n$.  Therefore we can apply the maximum-modulus principle to the reciprocal of $f^n_c(0)$ as a function of $c$ on this domain, and we see that the minimum value of $|f^n_c(0)|$ must be obtained on the boundary of $D_n$.  By our choice of $R$, this minimum is in fact obtained on the boundary of some $D(k)$; i.e. $|f^n_c(0)|$ as a function of $c$ is bounded below on $D_n$ by the value of $|f^n_c(0)|$ when $c$ is chosen such that 0 lies in an attracting basin of a point $a$ with period $k$ dividing $n$, with multiplier of modulus $\rho$.  Thus it suffices to provide a bound for these boundary $c$.\\

Suppose $c$ is a parameter such that 0 is in the basin of attraction of a point $a$ of exact period $n$ and multiplier $\rho$, and denote the immediate basin of attraction of $a$ by $V$.  We have a conformal isomorphism $\phi: \mathbb{D} \rightarrow V$; choose coordinates so that $\phi(0) = a$.  Write $g := \phi^{-1} \circ f^n_c \circ \phi.$ \\

Note that $f^n_c(z)$ is a proper map on $V$ and has well-defined degree.  By the chain rule, 0 is the only critical point of $f^n_c(z)$ which lies inside of $V$, and it has ramification index 2.  Therefore, since $V$ is simply connected, the Riemann-Hurwitz formula implies that $f^n_c$ is a degree 2 self-map of $V$.  So $g(w)$ is a proper, holomorphic, degree 2 map of the unit disk to itself which fixes 0.  Therefore $g$ is a Blaschke product:

$$g(w) = \lambda w \cdot \frac{\alpha - w}{1 - \bar{\alpha}w}$$
for some $\alpha \in \mathbb{D}$ and $|\lambda| = 1$.  But computing the derivative of $g$, we see that 

$$\rho = |g'(0)| = |\lambda| |\alpha| = |\alpha|.$$ 
Thus, applying a rotational coordinate change if necessary, we may assume that $\alpha = \rho,$ and so 

$$g(w) = \lambda w \cdot \frac{\rho - w}{1 - \rho w}.$$

We wish to use this correspondence to find an upper bound for the ratio $\frac{|a|}{|f^n_c(0)|}$, which along with a lower bound for $|a|$ will provide the desired lower bound for $|f^n_c(0)|$.  We normalize $\phi$, by defining $\psi: \mathbb{D} \rightarrow \mathbb{C}$ to be

$$\psi(w) = \frac{\phi(w) - a}{\phi'(0)}.$$

Let $p = \phi^{-1}(0)$; since $p$ is the critical point of $g$, we find $p = \frac{1 - \sqrt{1-\rho^2}}{\rho},$ and so in particular our assumption that $0 < \rho < \frac{1}{4}$ implies $0 < p < 4 - \sqrt{15}$.  We have:
\begin{equation} \label{estimate}
\frac{|a|}{|f^n_c(0)|} \cdot \frac{|g(p)-p|}{|p|} = \frac{|a|}{|p|} \cdot \frac{|g(p)-p|}{|f^n_c(0)|} = \frac{|\phi(0) - \phi(p)|}{|p-0|} \cdot \frac{|g(p) - p|}{|\phi(g(p)) - \phi(p)|}.
\end{equation}

We recall the deep theorem of de Branges (see \cite{deBranges:APOT} or the excellent expository article \cite{Zorn:TBC}):\\

\begin{theorem} \label{dB} [de Branges] Suppose $\psi(w): \mathbb{D} \rightarrow \mathbb{C}$ is one-to-one, with $\psi(0) = 0$ and $\psi'(0) = 1$.  Then the coefficients of the power series expansion
$$\psi(w) = z + a_2z^2 + a_3z^3 + ...$$
satisfy $|a_n| \leq n$ for all $n \geq 2$. 
\end{theorem}

As a consequence of de Branges' theorem, we can use the equality above to provide an upper bound for 
$$\frac{|\phi(0) - \phi(p)|}{|p-0|} \cdot \frac{|g(p) - p|}{|\phi(g(p)) - \phi(p)|} = \frac{|\psi(p)|}{|p|} \cdot \frac{|g(p) - p|}{|\psi(g(p)) - \psi(p)|}:$$

\begin{corollary} \label{distortionbound}
$$\frac{|\psi(p)|}{|p|} \cdot \frac{|g(p) - p|}{|\psi(g(p)) - \psi(p)|} \leq 4.$$
\end{corollary}
 
\begin{proof} Since $p = \frac{1 - \sqrt{1-\rho^2}}{\rho},$ we can compute $|g(p)| < p^2.$  Considering the power series expansion of $\psi$, we have
$$\frac{|\psi(g(p)) - \psi(p)|}{|g(p) - p|} \geq 1 - \sum_{n \geq 2} |a_n| \cdot |g(p)^{n-1} + g(p)^{n-2}p + ... + p^{n-1}|.$$
Therefore the theorem of de Branges and the estimate $|g(p)| < p^2$ yield
$$\frac{|\psi(g(p)) - \psi(p)|}{|g(p) - p|} \geq 1 - \sum_{n \geq 2} np^{n-1} \cdot \frac{p^n - 1}{p-1}.$$
Computation of the power series on the right gives
$$\frac{|\psi(g(p)) - \psi(p)|}{|g(p) - p|} \geq 1 - \frac{p(p^4-p^3-2p^2+3p+2)}{(1-p^2)^2} > \frac{1}{2}$$
with the right-hand inequality holding because $0 < p < 4 - \sqrt{15}.$ \\

On the other hand, we have 
$$\frac{|\psi(p)|}{p} \leq \sum_{n \geq 1} |a_n| p^{n-1} \leq \frac{1}{(1-p)^2} < 2,$$
and combining the two inequalities completes the proof of the corollary.
\end{proof}

Since in addition we have 
$$\frac{|g(p)-p|}{|p|} = |\frac{g(p)}{p} - 1| \geq 1 - |\frac{g(p)}{p}| \geq 1 - p > \frac{1}{2},$$
Corollary \ref{distortionbound} and Equation (\ref{estimate}) give
$$\frac{|a|}{|f^n_c(0)|} \leq 8.$$
By the chain rule, we have 

$$\rho = |(f^n_c)'(a)| = \prod_{0 \leq k < n}| f'(f^k(a))| = 2^n \prod_{0 \leq k < n} |f^k(a)|.$$
Since 0 is in the basin of attraction of $a$, $c$ lies in the Mandelbrot set and thus has modulus at most 2.  Consequently, $|z| >2 \Rightarrow |f_c(z)| = |z| \cdot |z + \frac{c}{z}| > |z|$, and so since $a$ is periodic, each iterate of $a$ has modulus bounded above by 2.  So we conclude that 
$$|a| \geq \rho \cdot \frac{1}{2^{2n-1}}.$$
Therefore 
$$|f_c^n(0)| \geq \rho \cdot \frac{1}{2^{2n+2}}.$$
One can use the same methods in the situation when $a$ has exact period $k \mid n$ to obtain the same lower bound, completing the proof of the theorem.
\end{proof}

Having achieved a lower bound for $|f_c^n(0)|$, we can now prove Theorem \ref{CA}.  We note that the choice of $\rho_n = \frac{1}{2^{2^{n-2}}}$ in the Theorem is the minimal value to achieve the tightest possible Zsigmondy result.

\begin{proof} [Proof of Theorem \ref{CA}] Define $D(n)$ as above for each $n$ with $\rho_n = \frac{1}{2^{2^{n-2}}}$, and write

$$S := \mathbb{C} - \bigcup_{n \in \mathbb{N}} D(n).$$
Suppose that $c = \frac{a}{b} \in S \cap (-2, -1)$.  Then if $a_n$ fails to have a primitive prime divisor for $n \geq 3,$ we have (as shown in preceding sections)

$$2 \log{|f^n_c(0)|} + 2^n \log{b} \leq 2 \omega(n) \log{|c|} + \sum_{q | n} 2^{\frac{n}{q}} \log{b},$$
where the sum is taken over distinct primes $q$ dividing $n$, and $\omega(n)$ is the number of distinct primes dividing $n$.  Write $s_2(n) = \sum_{q | n} 2^{\frac{n}{q}}.$  Then $|c| < 2$ and the lower bound for $|f^n_c(0)|$ obtained above yields 

$$(2^n - s_2(n)) \log{b} < (2 \omega(n) + 4n + 4 + 2^{n-1}) \log{2}.$$
Since $b \geq 2,$ this is false for all $n \geq 7.$  In fact for $b$ sufficiently large ($b \geq 13$ will suffice), this will be false for all $n \geq 3$, and the remaining finite number of cases can be checked to achieve the theorem.
\end{proof}

{\bf Acknowledgements.} I thank Laura DeMarco and Ramin Takloo-Bighash for valuable discussions crucial to this work, and the referees for suggesting improvements to the exposition.  I also thank Xavier Buff for his suggestions that led to \S 6, and Yann Bugeaud for helpful comments and references.  I am indebted to Patrick Ingram and Joe Silverman, whose paper \cite{IngramSilverman:PDIA} and related writings inspired this work, and the organizers of the Arizona Winter School 2010, where I was introduced to these questions.

\bibliography{research}{}
\bibliographystyle{plain}

\end{document}